\documentclass[11pt,reqno]{amsart}

\usepackage{amsfonts, amssymb, amscd}
\usepackage{amsmath}
\usepackage{verbatim}
\usepackage{amssymb}
\usepackage{mathrsfs}
\usepackage{graphicx}
\usepackage{cite}
\usepackage[all]{xy}
\usepackage{tikz}
\usepackage[toc,page]{appendix}
\usepackage{tikz-cd}

\usepackage{graphicx}
\usepackage{hyperref}
\usepackage{float}

\def\ii {{\mathrm i}}
\def\OO{{\mathcal O}}

\def\C {{\mathbb C}}
\def\Q {{\mathbb Q}}

\def\Z {{\mathbb Z}}
\def\P {{\mathbb P}}

\def\FPP {{\P^2_{fake}}}

\theoremstyle{definition}

\newtheorem{theorem}{Theorem}[section]

\newtheorem{proposition}[theorem]{Proposition}

\newtheorem{notation}[theorem]{Notation}

\newtheorem{remark}[theorem]{Remark}

\numberwithin{equation}{section}
\begin{document}
\title{Finding equations of the fake projective plane  $(C18,p=3,\{2I\})$}

\begin{abstract}
We find explicit equations of a new pair of fake projective planes, labeled by  $(C18,p=3,\{2I\})$ in the Cartwright-Steger classification. Our method involves starting with known equations of a commensurable fake projective plane $(C18,p=3,\emptyset,d_3 D_3)$ and working through a chain of cyclic covers and quotients to get to the new one. 
\end{abstract}

\author{Lev Borisov}
\address{Department of Mathematics\\
Rutgers University\\
Piscataway, NJ 08854} \email{borisov@math.rutgers.edu}

\author{Bojue Wang}
\address{Department of Mathematics\\
Rutgers University\\
Piscataway, NJ 08854} \email{bw391@math.rutgers.edu}

\maketitle

\tableofcontents
\section{Introduction.}\label{intro}
Complex projective algebraic surfaces $X$ are classified according to their Kodaira dimension $\kappa(X)$. The case $\kappa(X)=2$, when $X$ is a surface of general type, is arguably the least understood. Among such surfaces one is especially interested in those with small Hodge numbers, in particular $h^{1,0}(X)=h^{2,0}(X)=0$, see for example the review \cite{BCPsurvey}. These are further separated by the square of their canonical class $K_X^2$, and the extreme case $K_X^2=9$ occurs for the so-called fake projective planes (from here on called FPPs), which are characterized by having the same Hodge numbers as the usual projective plane $\C\P^2$.

\medskip
First example of an FPP was given by Mumford in \cite{Mumford}, using the method of $2$-adic uniformization. Research over the next several decades by multiple authors (see in particular \cite{CS,PY,Ishida,Keum-7,KeumQ,KK,Kl}) led to the full classification of all FPPs as quotients of the complex $2$-ball $\mathcal B^2$ by explicit finitely presented arithmetic groups in \cite{CS}. There are exactly $100$ FPPs up to isomorphism, gathered into $50$ pairs of complex conjugate surfaces. 

\medskip 
Unfortunately, a ball quotient description does not directly lead to any explicit equations of an FPP in its embedding into a projective space.
First such equations of an FPP in its bicanonical embedding into $\C\P^9$ were found in \cite{B-Keum}. Since then, L.B. and multiple coauthors have found explicit equations of ten more complex conjugate pairs of FPPs \cite{B21,BBF,BF,BJL}. The current paper continues this program by finding explicit equations of an FPP labeled by $(C18,p=3,\{2I\})$ in the Cartwright-Steger classification.  Our method involves starting from known equations of a commensurable FPP and using it to produce the equations of a new FPP. Even though it is similar to the method of \cite{BF} and \cite{BJL}, we had to overcome significant computational difficulties that arose because of the large degree of the common Galois cover of the two FPPs.

\subsection{General description of the process.}
The FPP indexed by $(C18,p=3,\{2I\})$ is known to be commensurable to the FPP indexed by 
$(C18,p=3,\emptyset,d_3 D_3)$ in Cartwright-Steger list whose equations were found in \cite{BBF} several years ago.
For the remainder of the paper, we will adopt the following notation.

\begin{notation}
We write $\FPP$ for the fake projective plane in the pair $(C18,p=3,\emptyset,d_3 D_3)$ whose equations were found in \cite{BBF}.
We write  $\widehat\FPP$ for the commensurable plane in the class $(C18,p=3,\{2I\})$ whose equations we find in this paper.
\end{notation}

The following results follow from the computations of Cartwright and Steger and additional GAP \cite{GAP}  calculations we did in the file GAPdataAll of \cite{BW+}.
The fake projective plane  $\FPP$ has an automorphism group $C_3\times C_3$ and $\widehat\FPP$ is a non-Galois degree $9$ cover of 
$\FPP/C_3\times C_3$. In the other direction, there is a surface which we denote by $72.\FPP$ which is a common Galois cover of $\FPP$ and $\widehat\FPP$.

\begin{proposition}
There is a surface $72.\FPP$ with an automorphism group $G_{648}$ of order $648 = 2^3 3^4$ which is isomorphic to the direct product of 
$C_3$ and the semidirect product of $SL(2,\Z/3\Z)$ and $C_3\times C_3$ (with the canonical action of the former on the latter)
$$
G_{648} \cong  C_3 \times \left( SL(2,\Z/3\Z) \ltimes (C_3\times C_3)\right).
$$
The fake projective planes $\FPP$ and 
 $\widehat\FPP$ are the quotients of $72.\FPP$ by the normal subgroup of  $G_{648}$ of order $72$
$$
G_{72}=\{1\}\times  \left( Q_8 \ltimes (C_3\times C_3)\right)
$$
and the nonnormal subgroup of  $G_{648}$ of order $72$
$$
 \widehat G_{72}=C_3\times  \left( SL(2,\Z/3\Z)  \times \{1\} \right),
$$
respectively.
Here $Q_8$ is the normal 2-Sylow subgroup of $SL(2,\Z/3\Z)$, isomorphic to the quaternion group.
\end{proposition}

\begin{proof}
This is a result of the GAP computation in GAPdataAll, see \cite{BW+}.
\end{proof}

In this paper 
we consider the following diagram of surfaces and Galois covers.
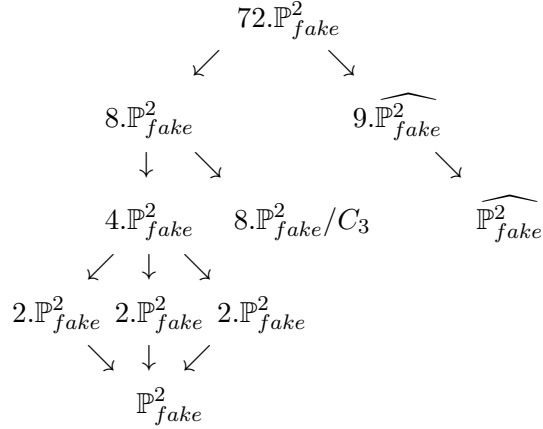
\begin{align*}
& \hskip 50pt 72.\FPP \\
&\hskip 30pt\swarrow\hskip 40pt \searrow\\
&8.\FPP       \hskip 60pt            9.\widehat\FPP\\
&\hskip 10pt\downarrow       \hskip 15pt  \searrow   \hskip 80pt        \searrow\\
&4.\FPP \hskip 15pt 8.\FPP /C_3    \hskip 40pt            \widehat\FPP\\
&\hskip -10pt \swarrow \hskip 10pt\downarrow    \hskip10pt\searrow \\
&\hskip -35pt 2.\FPP \hskip 5pt 2.\FPP  \hskip 5pt  2.\FPP\\
&\hskip -10pt \searrow \hskip 10pt\downarrow    \hskip10pt\swarrow \\
&\hskip 12 pt \FPP
\end{align*}
Here the surface $8.\FPP$ is the quotient of $72.\FPP$ by the normal subgroup
$$
\{1\}\times  \left( \{1\} \ltimes (C_3\times C_3)\right)
$$
of $G_{648}$. This surface is an unramified $Q_8$ Galois cover of $\FPP$
and $4.\FPP$ and $2.\FPP$ correspond to the center and three index two subgroups of $Q_8$, respectively. Three different covers 
$2.\FPP\to \FPP$ correspond to $2$-torsion line bundles on $\FPP$ that are permuted by an order three automorphism of $\FPP$.
The surface $8.\FPP /C_3$ is a singular quotient of $8.\FPP$ by the image of the central $C_3$ of $G_{648}$ (the first $C_3$ factor). It is used in intermediate calculations to get enough points on $72.\FPP$ with high accuracy. The surface $9.\widehat\FPP$ is the quotient of $72.\FPP $
 by the 2-Sylow subgroup of $G_{648}$
$$
 \{1\}\times  \left( Q_8 \ltimes \{1\}\right).
$$
It is a $C_3\times C_3$ unramified Galois cover of  $\widehat\FPP$, which we used to simplify the equations of $\widehat\FPP$.

\medskip
The method of the paper is to start with the known equations of $\FPP$, make our way up to $72.\FPP$ and then take invariants to get down to 
 $\widehat\FPP$. We start with equations of $\FPP$ found in \cite{BBF}. They describe the image of $\FPP$ in its bicanonical embedding into $\C\P^9$ as being cut out by $84$ cubic equations in $10$ variables. The coefficients (in $\Z[\sqrt{-2}]$) are about $100$ digits long. There is an explicit $C_3\times C_3$ action on $\C\P^9$ that gives automorphism of $\FPP$.

\medskip
{\bf Step 1.} We find a nonreduced linear cut on  $\FPP$  which corresponds to the square of a $2$-torsion element of the Picard group of $\FPP$.
This is done by first running an exhaustive search modulo $73$, then lifting the nonreduced cut it to $73$-adics to high enough accuracy and finally recognizing the resulting coefficients as algebraic numbers.

\medskip
{\bf Step 2.} We use the nonreduced linear cut above to simplify the equations of $\FPP$ by picking a better basis of $H^0(\FPP, 2K)$. We also choose to make it a basis of $C_3\times C_3$ eigenvectors. The downside is that the equations now have coefficients in $\Q(\sqrt{-2},\sqrt{-3})$ which will be the case for most of the process.

\medskip
{\bf Step 3.} We use the same cut to find equations of the double cover $2.\FPP$ in its putative bicanonical embedding into $\C\P^{19}$. It is cut out by $100$ quadratic equations in $20$ variables.\footnote{We do not actually verify that it is an embedding, but it is highly likely}
We also use the known automorphism group of $\FPP$ to construct $4.\FPP$. We do not try to compute the equations of its putative bicanonical embedding, as there would be too many of them, but we develop a way of constructing points of its image in $\C\P^{39}$ with high accuracy. 

\medskip
{\bf Step 4.} We construct the double cover $8.\FPP$ of $4.\FPP$. In fact, we first construct its quotient $8.\FPP/C_3$ which is likely the image of the canonical map $8.\FPP\to \C\P^6$.\footnote{Again, we do not actually verify that this is the image, but it factors through it.} It is given by $4$ cubic and $58$ degree four equations in $7$ variables. We also find the action of $Q_8$ on 
 $8.\FPP/C_3$. We then find a way of constructing points with high accuracy in what is likely the bicanonical embedding of $8.\FPP$ into $\C\P^{79}$.

\medskip
{\bf Step 5.} In what is arguably the most delicate part of the calculation, the $C_3\times C_3$ cover $72.\FPP \to 8.\FPP$ is determined by finding a relation among certain bicanonical sections on  $8.\FPP$.  As a result, we find a basis of the $71$-dimensional space $H^0(72.\FPP,K)$ in terms of algebraic functions in $H^0(8.\FPP,2K)$ and $H^0(8.\FPP,K)$. We also find the values of the elements of this basis  on some random points, with very high accuracy. We find the action of $G_{648}$ on this $71$-dimensional space, in particular we find the action of the subgroup
$$
 \widehat G_{72}=C_3\times  \left( SL(2,\Z/3\Z)  \times \{1\} \right).
$$
We compute the linear invariants of the $Q_8$ action on the dimension $71$ space to get points in the putative canonical embedding of 
$9.\widehat\FPP$ into $\C\P^7$.  We also compute the quadratic invariants of the action of $G_{72}$. We use $C_3\times C_3$ invariant products of elements of $H^0(9.\widehat\FPP,K)$   to compute special curves on $\widehat \FPP$.

\medskip
{\bf Step 6.} Looking at the intersections of the above special curves, we find a basis of $H^0(\widehat\FPP,2K)$ where the equations are defined over the smaller field $\Q(\sqrt{-2})$ 
and  have smaller coefficients.

\medskip
{\bf Step 7.} We use the usual methods to verify that the equations we obtain  indeed cut out a fake projective plane.

\subsection{Disclaimers, acknowledgements and further directions.}
We describe most of our surfaces in terms of multiple points in the images of  their maps into a projective space, computed with high accuracy (hundreds to thousands of decimal digits).
Thus our intermediate calculations cannot be deemed fully rigorous, which necessitates an eventual verification that the surface we obtain is indeed an FPP in a bicanonical embedding.
We have use GAP, Magma and Mathemaica \cite{GAP, Mag, Mat} with most of the computations performed in the latter.

\medskip
We believe that our results may allow us to eventually compute all three remaining pairs of FPPs that are commensurable to $\FPP$ and $\widehat\FPP$. They are labeled by
$(C18,p=3,\{2\},D_3)$,  $(C18,p=3,\{2\},(dD)_3)$ and $(C18,p=3,\{2\},(d^2D)_3)$ in Cartwright-Steger classification.

\section{Technical details: group-theoretic calculations.} 
In this section we collect the results of GAP calculations and related computations of characters of finite group representations.

\begin{proposition}\label{torsion}
The torsion subgroups of the Picard groups of the covers of fake projective planes above are given by the following table.
\begin{equation}
\begin{array}{c|c|c|c|c|c}
\FPP&4.\FPP &8.\FPP &72.\FPP&9.\widehat\FPP&\widehat\FPP\\
\hline&&&&&\\[-.5em]
     C_2^2\times C_{13} &   C_2^8\times C_3^2 &    C_3^2\times C_{13}         & C_2^8\times C_3\times C_{13}           &    C_2^3\times  C_3 \times C_{13}              & C_2 \times C_3^2
\end{array}
\end{equation}
\end{proposition}

\begin{proof}
Computed by GAP.
\end{proof}

Recall that there is a divisor class $H$ on $\FPP$ so that its canonical class is $K=3H$. We abuse the notation and use $K$ and $3H$ interchangeably throughout the rest of the paper to denote the canonical line bundles (or divisor classes) on various surfaces. Note that the corresponding line bundles come with a natural linearization with respect to the automorphism group of the surface.

\begin{proposition}\label{2.2}
We have the following dimensions of the spaces of global sections of the canonical and the bicanonical invertible sheaves $\OO(3H)$ and $\OO(6H)$ on the surfaces mentioned above. 
$$
\begin{array}{l|c|c}
&\dim_\C H^0(\cdot, 3H)& \dim_\C H^0(\cdot, 6H)\\
\hline&&\\[-.8em]
\FPP,\widehat\FPP& 0 & 10 \\[.4em]
2.\FPP & 1 & 20 \\[.4em]
4.\FPP & 3 & 40 \\[.4em]
8.\FPP & 7 & 80 \\[.4em]
8.\FPP/C_3 & 7 & 32\\[.4em]
72.\FPP & 71 & 720 \\[.2em]
9.\widehat\FPP & 8 & 90 
\end{array}
$$
\end{proposition}

\begin{proof}
The case of $8.\FPP/C_3$ is special since this surfaces is singular, and it also does not cover a fake projective plane. We will treat it last.

\medskip
For an $n$-fold unramified cover $X$ of a fake projective plane, we have $\chi(X, kH) = \frac n2(k-1)(k-2)$ by the Riemann-Roch theorem. Since $K_X=3H$ is ample, for $k>3$ Kodaira vanishing theorem assures that $h^0(X,kH)=\chi(X, kH)$. Thus, $h^0(X,6H)=10 n$ which gives the values of the right column of the above table. Similarly, we have 
\begin{align*}
&h^0(X,3H) = \chi(3H) - h^2(X,3H) + h^1(X,3H) \\
&=  n - h^{2,2}(X) + h^{1,2}(X) = n - 1 + h^{1,0}(X).
\end{align*}
Since the fundamental groups of the above surfaces have finite abelianization by Proposition \ref{torsion}, we have $h^{1,0}(X) = 0$. This implies the values in the middle  column of the above table.

\medskip
The dimensions of $H^0(8.\FPP/C_3 ,3H)$ and $H^0(8.\FPP/C_3 ,6H)$ are the dimensions of the spaces of invariants for the central $C_3$ action on $8.\FPP$. Note that the action of the generator $g$ of this $C_3$ on $\FPP$ has $3$ fixed points of type $\frac 13(1, 2)$, see \cite{KeumQ}. 
Since $g$ also acts on $2.\FPP$, its action on the two-point preimage of every fixed point on $\FPP$ must be trivial, so $g$ has $6$ fixed points of the same type on $2.\FPP$. Similarly, its action on $4.\FPP$ has $12$ fixed points and its action on $8.\FPP$ has $24$ fixed points. Each of these singular points has the same contribution to $\sum_{i=0}^2(-1)^i {\mathrm {Tr}}(g, H^i(X, 3H))$ in the the Holomorphic Lefschetz formula. We know that this sum is $1$ for $X=\FPP$, therefore it is equal to $8$ for $X=8.\FPP$, i.e.
$$
\displaystyle\sum_{i=0}^2(-1)^i {\mathrm {Tr}}(g, H^i(8.\FPP, 3H)) = 8.
$$
Consequently, all sections of $H^0(8.\FPP,3H)$ are invariant with respect to the action of $g$, and the canonical map of $8.\FPP$ factors through the quotient  surface $8.\FPP/C_3$. Since $\sum_{i=0}^2(-1)^i {\mathrm {Tr}}(g, H^i(8.\FPP, kH))$ 
depends only on $k\hskip -4pt\mod 3$, we see that the action of $g$ on the $80$-dimensional space $H^0(8.\FPP,6H)$ has invariant subspace of dimension $32$ and two spaces of dimension $24$ each of eigenvalues $e^{\pm 2\pi \ii/3}$.
\end{proof}

In what follows, it will be important for us  to understand the representation of $G_{648}$ on the dimension $71$ space $H^0(72.\FPP,3H)$. 
According to GAP \cite{GAP}, the character table of $G_{648}$ is given 
by
$$
\hskip -4pt
\tiny
\renewcommand{\arraystretch}{0.9}
\setlength{\arraycolsep}{0.6pt}
\begin{array}{l | rrrrrrrrrrrrrrrrrrrrrrrrrrrrrr}
 &1a &3a &2a &4a & 3b & 3c& 6a & 3d & 3e &6b &3f & 3g&  6c& 12a & 3h&  3i& 6d & 3j & 3k &6e& 3l & 3m & 6f& 12b & 3n & 3o& 6g & 3p & 3q& 6h\\
 \hline
\chi_1&      1&  1&  1&  1&   1&   1&  1&   1&   1&  1&  1&   1&   1&   1&   1&   1&  1&   1&   1&  1&  1&   1&   1&   1&   1&   1&  1&   1&   1&  1\\
\chi_2&      1&  1&  1&  1&   1&   1&  1&   1&   1&  1&  \alpha&   \alpha&   \alpha&   \alpha&   \alpha&   \alpha&  \alpha&   \alpha&   \alpha&  \alpha& \bar\alpha&  \bar\alpha&  \bar\alpha&  \bar\alpha&  \bar\alpha&  \bar\alpha& \bar\alpha&  \bar\alpha&  \bar\alpha& \bar\alpha\\
\chi_3 &     1&  1&  1&  1&   1&   1&  1&   1&   1&  1& \bar\alpha&  \bar\alpha&  \bar\alpha&  \bar\alpha&  \bar\alpha&  \bar\alpha& \bar\alpha&  \bar\alpha&  \bar\alpha& \bar\alpha&  \alpha&   \alpha&   \alpha&   \alpha&   \alpha&   \alpha&  \alpha&   \alpha&   \alpha&  \alpha\\
\chi_4  &    1&  1&  1&  1&   \alpha&   \alpha&  \alpha&  \bar\alpha&  \bar\alpha& \bar\alpha&  1&   1&   1&   1&   \alpha&   \alpha&  \alpha&  \bar\alpha&  \bar\alpha& \bar\alpha&  1&   1&   1&   1&   \alpha&   \alpha&  \alpha&  \bar\alpha&  \bar\alpha& \bar\alpha\\
\chi_5  &    1&  1&  1&  1&  \bar\alpha&  \bar\alpha& \bar\alpha&   \alpha&   \alpha&  \alpha&  1&   1&   1&   1&  \bar\alpha&  \bar\alpha& \bar\alpha&   \alpha&   \alpha&  \alpha&  1&   1&   1&   1&  \bar\alpha&  \bar\alpha& \bar\alpha&   \alpha&   \alpha&  \alpha\\
\chi_6  &    1&  1&  1&  1&   \alpha&   \alpha&  \alpha&  \bar\alpha&  \bar\alpha& \bar\alpha&  \alpha&   \alpha&   \alpha&   \alpha&  \bar\alpha&  \bar\alpha& \bar\alpha&   1&   1&  1& \bar\alpha&  \bar\alpha&  \bar\alpha&  \bar\alpha&   1&   1&  1&   \alpha&   \alpha&  \alpha\\
\chi_7  &    1&  1&  1&  1&  \bar\alpha&  \bar\alpha& \bar\alpha&   \alpha&   \alpha&  \alpha& \bar\alpha&  \bar\alpha&  \bar\alpha&  \bar\alpha&   \alpha&   \alpha&  \alpha&   1&   1&  1&  \alpha&   \alpha&   \alpha&   \alpha&   1&   1&  1&  \bar\alpha&  \bar\alpha& \bar\alpha\\
\chi_8&      1&  1&  1&  1&   \alpha&   \alpha&  \alpha&  \bar\alpha&  \bar\alpha& \bar\alpha& \bar\alpha&  \bar\alpha&  \bar\alpha&  \bar\alpha&   1&   1&  1&   \alpha&   \alpha&  \alpha&  \alpha&   \alpha&   \alpha&   \alpha&  \bar\alpha&  \bar\alpha& \bar\alpha&   1&   1&  1\\
\chi_9  &    1&  1&  1&  1&  \bar\alpha&  \bar\alpha& \bar\alpha&   \alpha&   \alpha&  \alpha&  \alpha&   \alpha&   \alpha&   \alpha&   1&   1&  1&  \bar\alpha&  \bar\alpha& \bar\alpha& \bar\alpha&  \bar\alpha&  \bar\alpha&  \bar\alpha&   \alpha&   \alpha&  \alpha&   1&   1&  1\\
\chi_{10} &    2&  2& -2&  .&  -1&  -1&  1&  -1&  -1&  1&  2&   2&  -2&   .&  -1&  -1&  1&  -1&  -1&  1&  2&   2&  -2&   .&  -1&  -1&  1&  -1&  -1&  1\\
\chi_{11}&     2&  2& -2&  .&  -\alpha&  -\alpha&  \alpha& -\bar\alpha& -\bar\alpha& \bar\alpha&  2&   2&  -2&   .&  -\alpha&  -\alpha&  \alpha& -\bar\alpha& -\bar\alpha& \bar\alpha&  2&   2&  -2&   .&  -\alpha&  -\alpha&  \alpha& -\bar\alpha& -\bar\alpha& \bar\alpha\\
\chi_{12}&     2&  2& -2&  .& -\bar\alpha& -\bar\alpha& \bar\alpha&  -\alpha&  -\alpha&  \alpha&  2&   2&  -2&   .& -\bar\alpha& -\bar\alpha& \bar\alpha&  -\alpha&  -\alpha&  \alpha&  2&   2&  -2&   .& -\bar\alpha& -\bar\alpha& \bar\alpha&  -\alpha&  -\alpha&  \alpha\\
\chi_{13}&     2&  2& -2&  .&  -\alpha&  -\alpha&  \alpha& -\bar\alpha& -\bar\alpha& \bar\alpha& \bar\beta&  \bar\beta& -\bar\beta&   .& -\bar\alpha& -\bar\alpha& \bar\alpha&  -1&  -1&  1&  \beta&   \beta&  -\beta&   .&  -1&  -1&  1&  -\alpha&  -\alpha&  \alpha\\
\chi_{14} &    2&  2& -2&  .& -\bar\alpha& -\bar\alpha& \bar\alpha&  -\alpha&  -\alpha&  \alpha&  \beta&   \beta&  -\beta&   .&  -\alpha&  -\alpha&  \alpha&  -1&  -1&  1& \bar\beta&  \bar\beta& -\bar\beta&   .&  -1&  -1&  1& -\bar\alpha& -\bar\alpha& \bar\alpha\\
\chi_{15} &    2&  2& -2&  .& -\bar\alpha& -\bar\alpha& \bar\alpha&  -\alpha&  -\alpha&  \alpha& \bar\beta&  \bar\beta& -\bar\beta&   .&  -1&  -1&  1& -\bar\alpha& -\bar\alpha& \bar\alpha&  \beta&   \beta&  -\beta&   .&  -\alpha&  -\alpha&  \alpha&  -1&  -1&  1\\
\chi_{16} &    2&  2& -2&  .&  -\alpha&  -\alpha&  \alpha& -\bar\alpha& -\bar\alpha& \bar\alpha&  \beta&   \beta&  -\beta&   .&  -1&  -1&  1&  -\alpha&  -\alpha&  \alpha& \bar\beta&  \bar\beta& -\bar\beta&   .& -\bar\alpha& -\bar\alpha& \bar\alpha&  -1&  -1&  1\\
\chi_{17} &    2&  2& -2&  .&  -1&  -1&  1&  -1&  -1&  1& \bar\beta&  \bar\beta& -\bar\beta&   .&  -\alpha&  -\alpha&  \alpha&  -\alpha&  -\alpha&  \alpha&  \beta&   \beta&  -\beta&   .& -\bar\alpha& -\bar\alpha& \bar\alpha& -\bar\alpha& -\bar\alpha& \bar\alpha\\
\chi_{18}&     2&  2& -2&  .&  -1&  -1&  1&  -1&  -1&  1&  \beta&   \beta&  -\beta&   .& -\bar\alpha& -\bar\alpha& \bar\alpha& -\bar\alpha& -\bar\alpha& \bar\alpha& \bar\beta&  \bar\beta& -\bar\beta&   .&  -\alpha&  -\alpha&  \alpha&  -\alpha&  -\alpha&  \alpha\\
\chi_{19} &    3&  3&  3& -1&   .&   .&  .&   .&   .&  .&  3&   3&   3&  -1&   .&   .&  .&   .&   .&  .&  3&   3&   3&  -1&   .&   .&  .&   .&   .&  .\\
\chi_{20} &    3&  3&  3& -1&   .&   .&  .&   .&   .&  .&  \gamma&   \gamma&   \gamma&  -\alpha&   .&   .&  .&   .&   .&  .& \bar\gamma&  \bar\gamma&  \bar\gamma& -\bar\alpha&   .&   .&  .&   .&   .&  .\\
\chi_{21}&     3&  3&  3& -1&   .&   .&  .&   .&   .&  .& \bar\gamma&  \bar\gamma&  \bar\gamma& -\bar\alpha&   .&   .&  .&   .&   .&  .&  \gamma&   \gamma&   \gamma&  -\alpha&   .&   .&  .&   .&   .&  .\\
\chi_{22}&     8& -1&  .&  .&  -1&   2&  .&  -1&   2&  .&  8&  -1&   .&   .&  -1&   2&  .&  -1&   2&  .&  8&  -1&   .&   .&  -1&   2&  .&  -1&   2&  .\\
\chi_{23}&     8& -1&  .&  .& -\bar\alpha&   \beta&  .&  -\alpha&  \bar\beta&  .&  8&  -1&   .&   .& -\bar\alpha&   \beta&  .&  -\alpha&  \bar\beta&  .&  8&  -1&   .&   .& -\bar\alpha&   \beta&  .&  -\alpha&  \bar\beta&  .\\
\chi_{24} &    8& -1&  .&  .&  -\alpha&  \bar\beta&  .& -\bar\alpha&   \beta&  .&  8&  -1&   .&   .&  -\alpha&  \bar\beta&  .& -\bar\alpha&   \beta&  .&  8&  -1&   .&   .&  -\alpha&  \bar\beta&  .& -\bar\alpha&   \beta&  .\\
\chi_{25} &    8& -1&  .&  .& -\bar\alpha&   \beta&  .&  -\alpha&  \bar\beta&  .&  \delta&  -\alpha&   .&   .&  -1&   2&  .& -\bar\alpha&   \beta&  .& \bar\delta& -\bar\alpha&   .&   .&  -\alpha&  \bar\beta&  .&  -1&   2&  .\\
\chi_{26} &    8& -1&  .&  .&  -\alpha&  \bar\beta&  .& -\bar\alpha&   \beta&  .& \bar\delta& -\bar\alpha&   .&   .&  -1&   2&  .&  -\alpha&  \bar\beta&  .&  \delta&  -\alpha&   .&   .& -\bar\alpha&   \beta&  .&  -1&   2&  .\\
\chi_{27} &    8& -1&  .&  .&  -1&   2&  .&  -1&   2&  .&  \delta&  -\alpha&   .&   .&  -\alpha&  \bar\beta&  .&  -\alpha&  \bar\beta&  .& \bar\delta& -\bar\alpha&   .&   .& -\bar\alpha&   \beta&  .& -\bar\alpha&   \beta&  .\\
\chi_{28}&     8& -1&  .&  .&  -1&   2&  .&  -1&   2&  .& \bar\delta& -\bar\alpha&   .&   .& -\bar\alpha&   \beta&  .& -\bar\alpha&   \beta&  .&  \delta&  -\alpha&   .&   .&  -\alpha&  \bar\beta&  .&  -\alpha&  \bar\beta&  .\\
\chi_{29} &    8& -1&  .&  .&  -\alpha&  \bar\beta&  .& -\bar\alpha&   \beta&  .&  \delta&  -\alpha&   .&   .& -\bar\alpha&   \beta&  .&  -1&   2&  .& \bar\delta& -\bar\alpha&   .&   .&  -1&   2&  .&  -\alpha&  \bar\beta&  .\\
\chi_{30} &    8& -1&  .&  .& -\bar\alpha&   \beta&  .&  -\alpha&  \bar\beta&  .& \bar\delta& -\bar\alpha&   .&   .&  -\alpha&  \bar\beta&  .&  -1&   2&  .&  \delta&  -\alpha&   .&   .&  -1&   2&  .& -\bar\alpha&   \beta&  .\\
\end{array}
$$
where we have 
$\alpha = w^2,~
\beta 
  = 2w ,~
\gamma 
  = 3w^2,~
\delta 
  =  8 w^2
$ for $w=e^{2\pi\ii/3}$.

\begin{proposition}\label{weights71}
The character of the representation of  $G_{648}$ on  the space $H^0(72.\FPP,3H)$ is given by 
\begin{equation}\label{char71}
\chi_{11}+\chi_{12}+\chi_{19} + \chi_{23} +  \chi_{24} +  \chi_{25} +  \chi_{26} + 
 \chi_{27} +  \chi_{28} +  \chi_{29} +  \chi_{30}.
\end{equation}
\end{proposition}

\begin{proof}
By the Holomorphic Lefschetz Formula, together with the trivial representation of  $G_{648}$ on $H^2(72.\FPP,3H)$, our representation must restrict to the regular representation both for $G_{72}$ and for $\widehat G_{72}$, since every nonidentity element of this group has no fixed points. We use GAP (GAPdataAll) to compute the restrictions 
of the characters to these groups and then Mathematica (Dim71rep.nb) to find the unique linear combination of characters of $G_{648}$ that has this property.
\end{proof}

\begin{proposition}\label{weights7}
The dimension $7$ subspace of $H^0(72.\FPP,3H)$ with character $\chi_{11}+\chi_{12}+\chi_{19}$ can be naturally identified with 
$$H^0(8.\FPP,3H) \cong H^0(8.\FPP/C_3,3H).$$ Similarly, $H^0(4.\FPP,3H)$ can be identified with $\chi_{19}$. 
\end{proposition}

\begin{proof}
The map  $72.\FPP\to 8.\FPP$ is a Galois cover, so $G_{648}$ acts on the space of holomorphic $2$-forms on $8.\FPP$, and the pullback map is compatible with the action. The subgroup $G_{72}$ is the kernel of the abelianization map $G_{648}\to C_3\times C_3$ and is thus built from the conjugacy classes $1a$, $2a$, $3a$ and $4a$. Then $\chi_{11}$, $\chi_{12}$ and $\chi_{19}$ are characterized by the property that they are invariant under $3a$, i.e. invariant under the normal subgroup $C_3\times C_3$ of $G_{72}$. Then $\chi_{19}$ is further characterized by having trivial action of the conjugacy class $2a$, which is the central involution of $Q_8$.
\end{proof}

\section{Technical details: computing the equations.}

In this section we comment  in more detail on the technical issues encountered in our process, as sketched in the Introduction.

\subsection{Step 1.}
Let $s$ be a nonzero element of $H^0(2.\FPP,3H)$. Then $s^2\in H^0(2.\FPP,6H)$ is invariant under the covering involution of the double cover $2.\FPP\to \FPP$ and is thus a pullback of an element of $H^0(\FPP,6H)$. Moreover, this section must be invariant with respect to the action of the central $C_3$. So to find $s^2$, we looked for nonreduced linear cuts of $\FPP$ in its bicanonical embedding, which are invariant under an action of a subgroup of its automorphism group (we did not a priori know which subgroup of ${\rm Aut}(\FPP)$ came from the central $C_3$). We first found such nonreduced cut modulo $73$ by a brute force search using Magma.\footnote{Unfortunately, $73$ was the smallest prime of the form $9k+1$ where the equations of \cite{BBF} gave a reduced surface with the correct Hilbert polynomial, and it took a considerable amount of time to go through all of the cases.} This calculation was entirely similar to the one in \cite{B21,BJLM,BL}. As a result,
we got 
\begin{align}\label{cut}
\begin{array}{c}
{\rm Cut}= Q_0 + 69 (Q_1 +  Q_2 +  Q_3) + 7 (Q_4 +  Q_5 + Q_6) + 62( Q_7 + Q_8 + 
 Q_9)
\end{array}
\end{align}
where $Q_i$ were the variables of the equations of $\FPP$ from \cite{BBF}.

\medskip
Our next goal was to lift the equation of the nonreduced cut \eqref{cut} from $\Z/73\Z$ to $\Z/73^k\Z$ for increasing powers of $k$. The previous method, used in the aforementioned papers was to find some points on the nonreduced cut, and enforce the condition of the cut being singular on (lifts of) the points as $k$ grows. However, this approach was unavailable in our case because there were no points on $X$ defined over $\Z/73\Z$. While we could have presumably worked  over a finite field extension, we found the following easier alternative approach.

\medskip
We used Magma to compute the ideal of the radical of the nonreduced cut of $\FPP$ modulo $73$. One of the equations was 
\begin{align*}
\begin{array}{rl}
F &= Q_3 Q_6 + 61 Q_4 Q_6 + 29 Q_5 Q_6 + 53 Q_6^2 + 9 Q_1 Q_7 + 18 Q_2 Q_7 + 
   42 Q_3 Q_7 \\
 &  + 32 Q_4 Q_7 + 15 Q_5 Q_7 + 9 Q_6 Q_7 + 11 Q_7^2 + 25 Q_1 Q_8 + 
   3 Q_3 Q_8 + 13 Q_4 Q_8 
   \\
   &+ 18 Q_5 Q_8 + 21 Q_6 Q_8 + 11 Q_7 Q_8 + 44 Q_8^2 + 
   49 Q_1 Q_9 + 63 Q_2 Q_9 + 53 Q_3 Q_9
   \\& + 12 Q_4 Q_9 + 26 Q_5 Q_9 + 12 Q_6 Q_9 + 
   51 Q_7 Q_9 + 68 Q_8 Q_9 + 44 Q_9^2
   \end{array}
\end{align*}
in the variables $Q_i$ of the bicanonical embedding of $\FPP$ modulo $73$. Then $F^2$ was in the ideal of the cut, and we wrote 
\begin{equation}\label{syzygy}
F^2 = \sum_{i=1}^{84} H_i E_i + {\rm Cut}\,R
\end{equation}
as polynomials in $Q_0,\ldots,Q_9$. Here $E_i$ are the equations of $\FPP$ (cubic in $Q$), $H_i$ are unknown linear combinations of $Q$, and $R$ is an unknown cubic polynomial in $Q$.
We originally computed a relation \eqref{syzygy} modulo $73$ and then lifted it modulo $73^k$ for increasing powers of $k$. At each step $k\to k+1$ we had a system of linear equations modulo $73$ on the corrections to the coefficients of $H_i$, $F$ and $R$. We used Mathematica to solve it, and picked the initial solution which was automatically taking case of making some coefficients zero. We went up to $73^{30}$ which gave a good approximation to coefficients of the nonreduced cut over the complex numbers.

\medskip
There is a standard way of guessing an algebraic number from its $p$-adic approximation. We used it to see that the nonreduced cut is given by
\begin{equation}\label{cutexact}
\begin{array}{l}
Q_0 - \frac{(-773 + 16\, \ii \sqrt 2)}{66449} (Q_1 + Q_2 + Q_3) - 
  W \frac{ (-50345 - 26294\, \ii \sqrt 2)}{132898 } (Q_4 + Q_5 + Q_6) 
  \\
  - 
\frac {(50345 + 26294\,\ii \sqrt 2 )}{132898 \,W} (Q_7 + Q_8 + Q_9)
  \end{array}
\end{equation}
where $W = (\frac 13 (2 - \ii \sqrt 2 ))^{\frac 13}$. The details of the above calculation are in the Mathematica file Step1.nb.

\subsection{Step 2.} 
We work out Steps 2 and 3 in the Mathematica file Steps23.nb. By simply scaling $Q$-s by the appropriate powers of $W$, we arranged the cut of \eqref{cutexact} to be defined over $\Q(\sqrt {-2})$ with equations of $\FPP$ still defined over this field. However, it was convenient for us to enlarge the field to $\Q(\sqrt {-2},\sqrt{-3})$ and to pick a basis of eigenvectors of the $C_3\times C_3$ action on $H^0(\FPP,6H)$, with the new variables called $R_0,\ldots, R_9$. We made the cut to be $R_1+R_4+R_7$, made one of the fixed points of a $C_3$ action to be 
$$
(0: 1: 1: 1: 0: 0: 0: 0: 0: 0)
$$
and made the tangent space to the cone over $\FPP$ at this point to be generated by 
$$
(0, 1, 1, 1, 0, 0, 0, 0, 0, 0), (0, 0, 0, 0, 1, 1, 1, 0, 0, 0), (0, 0, 0, 0, 0, 0, 0, 1, 1, 1).
$$
These  conditions fixed the basis of $R_i$, and the resulting equations of $\FPP$ had very small coefficients in $\Z[\sqrt {-2},\sqrt{-3}]$.

\subsection{Step 3.} 
The torsion subgroup of the Picard group of $\FPP$ is isomorphic to $C_{13}\times C_2\times C_2$, and the nontrivial two-torsion elements are acted upon by an order $3$ automorphism of $\FPP$ which scales $R_4,R_5,R_6$ by the primitive third root of unity $w=\frac 12(-1+\ii \sqrt 3)$ 
and scales $R_7,R_8,R_9$ by $w^2$. Each $2$-torsion element in the Picard group of $\FPP$ gives a nonreduced cut of it in the bicanonical embedding, which are thus
\begin{equation}\label{eq1}
R_1+R_4+R_7,~R_1 + w R_4 + w^2 R_7,~ R_1 + w^2 R_4 + w R_7,
\end{equation}
in the $R_i$ coordinates.

\medskip
The field of fractions of the ring $\bigoplus_{k\geq 0} H^0(4.\FPP,3kH)$ is obtained from that of  $\bigoplus_{k\geq 0} H^0(\FPP,6kH)$ by attaching the square roots of the linear forms \eqref{eq1}. The sections of $H^0(2.\FPP,6H)$ of the double cover can be then obtained as pullbacks of $R_i$ and as 
$$
\frac{F(R)}{\sqrt {R_1 + w R_4 + w^2 R_7}\,\sqrt {R_1 + w^2 R_4 + w R_7}}
$$
where $F(R)$ are quadratic polynomials in $R_i$ with the property that they are zero on the loci of zeros of $R_1 + w R_4 + w^2 R_7$ and 
$R_1 + w^2 R_4 + w R_7$. These have been computed and given the names $U_0,\ldots, U_{19}$ where $U_i=R_i$ for $0\leq i\leq 9$ form a basis of the subspace of invariants of the covering involution of $2.\FPP\to \FPP$, and the $U_{10},\ldots,U_{19}$ form a basis of the $(-1)$-eigenspace. Adding the $C_3$ translates of the latter gave us a basis $U_0,\ldots, U_{39}$ of $H^0(4.\FPP,6H)$ in terms of $R_i$ and the above square roots. We also extend the action of $C_3\times C_3$ to these $U_i$. 

\begin{remark}
We computed equations of the double cover $2.\FPP$, i.e. the relations among $U_0,\ldots,U_{19}$ and got the expected dimension $100$ space of these equations. We suspect that these quadratic equations cut out  $2.\FPP$ in its bicanonical embedding, but we did not try ascertain that (and it also may be beyond the reach of our hardware). We did not use these equations later in our computations.
\end{remark}

\subsection{Step 4.}
Step 4 takes a lot of work, and it is done in Step4.nb. 

\medskip 
We first 
recall that $8.\FPP$ is acted upon
by 
$$
C_3\times SL(2,\Z/3\Z)
$$
so that the quotient by the normal $2$-Sylow subgroup $Q_8$ of $SL(2,\Z/3\Z)$ induces the automorphism action of $C_3\times C_3$ on $\FPP$.  By our construction of Step 3, we have also lifted the action of $C_3\times C_3$ to act on  $4.\FPP$, which is the quotient of $8.\FPP$ by the central involution $\sigma\in Q_8$. Holomorphic Lefschetz formula allows one to figure out the action of $C_3\times SL(2,\Z/3\Z)$ on $H^0(8.\FPP,3H)$ and there exist, unique up to scaling, two elements $s_1$ and $s_2$ of $H^0(8.\FPP,3H)$ with the following properties. 

\begin{itemize}
\item Both $s_i$ are $(-1)$-eigenfunctions for the covering involution $\sigma$ of  $8.\FPP\to 4.\FPP$.
\item Both $s_i$ are invariant with respect to the central $C_3$.
\item Both $s_i$ are eigenfunctions for the action of the other $C_3$, one with weight $w$ and the other with weight $w^2$.
\end{itemize}

Consequently, $f_0 = s_1 s_2$, $f_1=s_1^2$ and $f_2= s_2^2$ are invariant with respect to $\sigma$ and are pullbacks of elements of $H^0(4.\FPP,6H)$ (i.e. linear combinations of $U_0,\ldots,U_{39}$) that satisfy 
\begin{equation}\label{eq2}
f_0^2 = f_1 f_2
\end{equation}
and have prescribed weights with respect to the $C_3\times C_3$ action on $U_i$. This resulted in a system of $26$ quadratic equations on $16$ unknown coefficients of $f_i$. After a fortunate choice of two additional scaling equations (since $s_i$ are only up to scaling, we can scale two out of three $f_i$), Mathematica readily solved the resulting equations in $14$ variables numerically and then recognized the results as good approximations to algebraic numbers in $\Q(\sqrt{-2},\sqrt{-3})$. 

\medskip
There is a two-dimensional subspace of $H^0(8.\FPP,3H)$ which is anti-invariant with respect to $\sigma$ and is invariant with respect to both $C_3$ groups. For an element $v$ of it, we knew where $v^2$ and $v s_1$ were, which allowed us to find it. This gave us seven linearly independent elements of $H^0(8.\FPP,3H)$, namely the square roots of \eqref{eq1}, and four other of the form $F_i(U)/\sqrt{f_1}$ for a solution $f_1$ of \eqref{eq2} and a linear function $F_i$ of $U_0,\ldots, U_{39}$. We denoted this basis by $V_0,\ldots, V_6$.

\medskip
By Proposition \ref{2.2}  sections $V_0,\ldots, V_6$ of $H^0(8.\FPP,3H)$ are all invariant with respect to the action of the central $C_3$ (this can also be seen by direct examination). As a consequence, they are pullbacks of the sections of the canonical line bundle on the singular surface $8.\FPP/C_3$ with $24$ singularities of type $A_2$. 
We found equations on $V_i$, namely a $4$-dimensional space of cubic equations and a $58$-dimensional space of quartic equations. These equations allowed us to find points on
the canonical image of $8.\FPP$ (or $8.\FPP/C_3$)  with high accuracy.

\medskip
The next step was to lift the action of $C_2\times C_2$ on $4.\FPP$ to an action of $Q_8$ on $8.\FPP$.
Specifically, this meant finding an order $4$ automorphism which lifted the order $2$ automorphism of $4.\FPP$. We also knew that its action on the dimension four subspace spanned $V_3,\ldots,V_6$ was traceless, and that the action permuted the points with $V_1=V_2=0$. Taken together, this information allowed us to find the desired order $4$ automorphism.

\medskip
Our next goal was to understand the space $H^0(8.\FPP, 6H)$.
We have $\dim H^0(8.\FPP, 6H)=80$, so it would be rather useless to try to compute equations among these, since solving systems of nonlinear equations in $80$ variables is well beyond the capabilities of our available hardware. Instead, we had to settle for being able to compute a lot of points in the image  $8.\FPP\to \C\P^{79}$ with high accuracy. The approach we took was to first compute points on  
$8.\FPP/C_3$ where we do have equations and then compute the values of elements $H^0(8.\FPP,6H)$ on them.

\medskip
We observed that $R_0,R_1,R_4,R_7$ can be easily written as degree two polynomials in $V_i$. In contrast, $R_2$ is not invariant under the central $C_3$ and can therefore not be written as a rational function in $V_i$. However, $R_2^3$ can  be written as a rational function in $R_0,R_1,R_4,R_7$ (namely as a ratio of a degree $12$ polynomial and a degree $9$ polynomial), and we computed it. Similarly, we computed rational functions in these variables for $R_2R_3,R_2R_6,R_2R_9,R_2^2R_5,R_2^2R_8$. Recall that $U_i$ for $0\leq i \leq 39$ are written as rational functions in $R_i$ and $V_0,V_1,V_2$. 
Therefore, for given (high accuracy) values of $V_i$, we can find three values for $R_2$ and then find values of the rest of $R_i$ and $U_0,\ldots,U_{39}$ for each of three values of $R_2$. 

\medskip
We then computed the subspace of $H^0(8.\FPP,6H)$ of sections that are anti-invariant with respect to the covering involution $\sigma$.
We did this by considering rational functions in $V$-s and $R_2$, of total $V$-degree $2$, which are zero on the curves $V_0+V_1+V_2=0$ and $V_3=0$, divided by $(V_0+V_1+V_2)V_3$. We first got a database of points on these two curves, and then computed vanishing conditions. The calculation was performed in Step4.nb and is split into three cases according to the character of the central $C_3$. Specifically, for the trivial character, we looked for degree four polynomials in $V_i$ which vanish at the aforementioned curves. For the other characters, we looked for linear combinations of products of quadratic polynomials in $V$ with some sections of $H^0(4.\FPP,6H)$ with the same  central weight.

\medskip
Afterwards, we computed the action of the two $C_3$-s (the central one and the chosen subgroup of $SL(2,\Z/3\Z)$) on the space $H^0(4.\FPP,6H)$ of dimension $80$. We picked an eigenbasis of it, denoted by $\tilde U_0, \ldots, \tilde U_{79}$. Finally, we computed points on $8.\FPP$ with accuracy of several thousand digits, in preparation for the next step.

\subsection{Step 5.} Naturally, this is the trickiest step of the whole paper, worked out in Step5.nb.

\medskip
The map $72.\FPP\to 8.\FPP$ is a Galois cover with the covering group $C_3\times C_3$, and we have a good understanding of $H^0(72.\FPP,3H)$ by 
Proposition \ref{weights71}.  In what follows, we will denote the corresponding subspaces of $H^0(72.\FPP,3H)$ as $H^0(72.\FPP,3H)_{11},\ldots, H^0(72.\FPP,3H)_{30}$, according to the index of the irreducible character.
Note that each of the $8$-dimensional irreps of $H^0(72.\FPP,3H)_{i}$ for $23\leq i\leq 30$ has one-dimensional eigenspaces for all non-trivial characters of the covering group. Indeed, all nonzero elements of this group are in the conjugacy class $3a$ and thus have trace $(-1)$. We also observe that each of these representations is acted upon by the central involution $\sigma$ of $Q_8$ which permutes $C_3\times C_3$ eigenspaces by inverting eigenvalues, because it corresponds to $(-{\rm Id})$ in $SL(2,\Z/3\Z)$.
The following proposition is the key to our approach.

\begin{proposition}\label{key}
Consider an order $3$ element $h$ of $SL(2,\Z/3\Z)$ and its action on the $3$-torsion  subgroup $C_3\times C_3$ of ${\rm Pic}(8.\FPP)$. Suppose that the 
character $(w,1)$ of the covering $C_3\times C_3$ corresponds to the eigenvector of $h$ in $C_3\times C_3$.
Let $f_1$ and $f_2 = \sigma(f_1)$ be a $(w,1)$-eigenvector and a  $(w^2,1)$-eigenvector for the covering $C_3\times C_3$ in the space $H^0(72.\FPP,3H)_{29}$, respectively.
Likewise, let $g_1$ and $g_2=\sigma(g_1)$ be an $(w,1)$- and $(w^2,1)$-eigenvectors  in the space $H^0(72.\FPP,3H)_{30}$. Then  $s_1 = f_1 f_2$, $s_2 = f_1 g_2$, $s_3 = g_1 f_2$, $s_4 = g_1 g_2 $
are invariant under the covering group and can be thought of as elements of $H^0(8.\FPP,6H)$. These sections $s_i$ have the following properties.
\begin{itemize}
\item
$
s_1s_4 = s_2 s_3
$ 
\item 
$\sigma(s_2) = s_3$
\item
Sections $s_1,s_2, s_3,s_4$ have weights $(w,1,1,w^2)$ respectively for the central $C_3$ action on $H^0(8.\FPP,6H)$. 
\item
The weights of $s_1,s_2,s_3,s_4$ for the action of $h\in SL(2,\Z/3\Z)$ are $(w^{2a},w^{a+b},w^{a+b},w^{2b})$ for some $a$ and $b$ in $\Z/3\Z$.
\end{itemize}
\end{proposition}

\begin{proof}
The first two statements are immediate from the construction. To prove the third statement, observe that the generator of the central $C_3$  has trace $8w^2$
in $\chi_{29}$ and $8w$ in $\chi_{30}$ (after an appropriate choice of generator or a switch of $\chi_{29}$ and $\chi_{30}$). Thus $f_i$ have eigenvalues $w^2$ and $g_i$ have eigenvalues $w$. 

\medskip
The last statement is the most delicate. Since $h$ preserves the corresponding element of the Picard group, its action preserves the corresponding eigenspaces of $H^0(72.\FPP,3H)_{29}$ and $H^0(72.\FPP,3H)_{30}$. Thus $f_i$ and $g_i$ are eigenvectors for its action, with eigenvalues $w^a$ and $w^b$ for some $a$ and $b$.
\end{proof}

\begin{remark}\label{notspecial}
There is nothing particularly special about using $\chi_{29}$ and $\chi_{30}$ in Proposition \ref{key}. In fact, $29$ can be replaced by $25$ or $27$ and $30$ can be replaced by $26$ or $28$. Since we do not know which values of $a$ and $b$ correspond to which subrepresentations, as we get a solution $(s_1,\ldots,s_4)$ we will not know exactly which subrepresentations they come from. 
\end{remark}

For each pair of values $(a,b)$, the conditions of Proposition \ref{key} can be translated into a system of polynomial equations on the coefficients of $s_i$ in the bases of the corresponding subspaces of $H^0(8.\FPP,6H)$. The number of variables is generally under $20$, and we were able to solve one the systems. Specifically, we solved it modulo $4363$, which is a large prime for which both $(-2)$ and $(-3)$ are quadratic residues, then lifted the solution to powers of $4363$ and finally used this $p$-adic approximation of solutions to realize them as algebraic numbers.

\medskip
Getting an equation of the form $s_1 s_4 = s_2 s_3$ is indicative of some additional divisor classes, given by $(s_1,s_3)$ and $(s_1,s_4)$. We computed the corresponding divisors 
and found that a third power can be written as a section of $H^0(8.\FPP,9H)$. Specifically, we were able to write it as a degree $3$ polynomial (called \emph{goodrr} in the Mathematica file Step5.nb) in $V_0,\ldots, V_6$, since a third power is also invariant with respect to the central $C_3$. In view of Remark \ref{notspecial}, we do not know precisely which irreducible subrepresentation the corresponding section $f_1\in H^0(72.\FPP,3H)$ lies in, but it is not important to us. Indeed, we know from Proposition \ref{torsion} that unramified triple covers of $8.\FPP$ come from $72.\FPP$, and we know that by adding $f_1$, and its $Q_8$-translates to the function field of the cone over $8.\FPP$, we will get the function field of the cone over  $72.\FPP$.

\medskip
More precisely, we computed a basis of a dimension $8$ subspace of elements of $H^0(8.\FPP,6H)$ which vanish on $f_1=0$ and were thus able to describe a set of $8$ linearly independent sections in $H^0(72.\FPP,3H)$ as $R/f_1$ for $R$ in this subspace. Then $Q_8$ translates of these forms, together with (pullbacks of) $V_i\in H^0(8.\FPP,3H)$ gave the basis of  $H^0(72.\FPP,3H)$. To be able to really compute values of 
elements of $H^0(72.\FPP,3H)$ on points of $72.\FPP$ we needed to be careful in identifying values of $Q_8$-translates  $f_i$ of $f_1$. While we knew their cubes, it was not clear which cubic roots had to be taken. This issue was solved by computing products $f_i f_j f_k$ which lie in $H^0(8.\FPP,9H)$ and using the values of the products to pick correct values of all but two $f_i$ (first two $f_i$ can be taking arbitrarily, each choice giving one of the preimage points of $C_3\times C_3$ cover $72.\FPP\to 8.\FPP$).

\medskip
In order to construct the surfaces $9.\widehat\FPP$ and $\widehat\FPP$ we found a lift of the action of $C_3\times SL(2,\Z/3\Z)$ from $8.\FPP$ to $72.\FPP$ by picking lifts of the generators. We then averaged over $Q_8$ to get values of sections of $H^0(9.\widehat\FPP,3H)$, called $W_1,\ldots W_8$. We similarly averaged over $C_3\times SL(2,\Z/3\Z)$ to get a basis of  $H^0(\widehat\FPP,6H)$, called $Z_0,\ldots, Z_9$. We computed equations on $Z_i$, which were $\dim 84$ space of cubics in $Z_i$, with coefficients in $\Q(\sqrt{-2},\sqrt{-3})$. In fact, we had to assume that the equations would lie in this field and still had to use several thousand digits of accuracy in our computation of points. We also computed the values of  four pairwise products of $W_i$ which lie in $H^0(\widehat\FPP,6H)$ which gave natural reducible linear cuts of $\widehat \FPP$ in its bicanonical embedding. These were used in the next step.

\subsection{Step 6.} At this point we had putative equations of $\widehat\FPP$ but the coefficients were large and were defined over $\Q(\sqrt{-2},\sqrt{-3})$. Both of these features made working with this surface difficult. We followed a rather ad hoc process which somewhat surprisingly allowed us to take care of both issues.

\medskip
First or all, for each pair of reducible cuts, found in Step 5, we computed their $36$ intersection points on $\widehat\FPP$. We speculated that $Z_0$ had to be defined over $\Q(\sqrt{-2})$, in the sense that there is a model of $\widehat\FPP$ over this field where $Z_0$ is defined over it. We normalized the $36$ points of intersections to have $Z_0=1$. Then we separated these $36$ points according to their field of definition. We added these points to get linear combinations of the basis dual to $Z_i$ with coefficients in the field $\Q(\sqrt{-2},\sqrt{-3})$. As the pairs varied, we ended up picking $21$ such points.
We also speculated that $Z_1$ should be defined over  $\Q(\sqrt{-2})$ and used natural linear combinations of the above $21$ points to get $13$ natural points in $\C^{10}$.
We then picked $10$ linearly independent ones and used a linear change of variables so that new sections were a dual basis. The resulting equations were indeed defined over  $\Q(\sqrt{-2})$, but the coefficients were up to $8\times 10^3$ digits long. The process was further refined by picking a small $\Q(\sqrt{-2})$-linear combination of the above $10$ points defined over $\Q(\sqrt{-2},\sqrt{-3})$. That led to equations in the new variables $Y_0,\ldots Y_9$ which were in $\Q(\sqrt{-2})$ and had coefficients a few hundred digits long. Finally, we traded the number of nonzero terms for the size of the coefficients by picking linear combinations of the equations via a lattice reduction algorithm. This led to the final output where the equations were only $20$ to $30$ decimal digits long, in $\Z[\sqrt{-2}]$. It seems plausible that one can reduce the coefficients further by picking a better basis of $H^0(\widehat\FPP,6H)$, but we were unable to do so.

\medskip
The details are in the file Step6.nb.

\subsection{Step 7.} The techniques of the previous steps used probabilistic approaches and approximate calculations, and the overall complexity of the code was also formidable. Fortunately, it is possible to verify that the surface we obtained is a fake projective plane by doing exact and relatively short calculations in Magma. We can then confidently identify it as $(C18,p=3,\{2I\})$. The method of verification that the surface is an FPP has not changed much since \cite{BF}. Specifically, we first observed that the surface $S$ in question has the correct Hilbert polynomial. Then we showed that it is smooth by picking three random minors of the Jacobian matrix and checking that adding them to the equations gives zero Hilbert polynomial over a finite field. For better or for worse, we used the same minors as in \cite{BL}, and it worked. We also computed the dimension of the cohomology spaces of the structure sheaf and the first cohomology space of the cotangent bundle. This allowed us to conclude that the surface is an FPP. Then it suffices to compute, as in \cite{B-Keum}, that $h^2(X,2K_X(-1))=0$ to show that our embedding is precisely the bicanonical one. As a slight improvement over previous approaches, we did the calculations entirely in Magma, as opposed to a mix of Magma and Macaulay2. The details are in Step7Magma (Hilbert polynomial of $\widehat\FPP$ over the number field) and Step7Magma4363 (the rest).

\end{document}